\documentclass[a4paper,11pt,reqno]{amsart}
\usepackage{amssymb,amsfonts,amsthm,mathtools,enumerate,verbatim,color,tikz,url,tcolorbox}
\usepackage{setspace}
\allowdisplaybreaks
\usepackage{prooftree}
\usepackage{multicol}
\usepackage[pagebackref=true]{hyperref}

\theoremstyle{plain}
\newtheorem{theorem}{Theorem}[section]
\newtheorem*{theorem*}{Theorem}
\newtheorem{proposition}[theorem]{Proposition}
\newtheorem*{proposition*}{Proposition}
\newtheorem{lemma}[theorem]{Lemma}

\theoremstyle{definition}

\newtheorem{example}[theorem]{Example}

\newtheorem{remark}[theorem]{Remark}

\newtheorem*{question}{Question}

\def\vp{\varphi}
\newcommand{\vP}{\varPi}
\newcommand{\vS}{\varSigma}
\newcommand{\vD}{\varDelta}

\DeclareMathOperator{\rank}{\mathrm{rk}}
\def\rul#1#2#3{\prooftree     #1 \justifies #2 \using{#3} \endprooftree}
\newlength{\tmplenA}
\newlength{\tmplenB}
\newcommand{\godel}[1]{\settowidth{\tmplenA}{$\ulcorner$}  
	\settoheight{\tmplenB}{$#1$}
	\setlength{\tmplenB}{0.9\tmplenB}
	\addtolength{\tmplenB}{-\tmplenA}
	\raisebox{\tmplenB}{$\ulcorner$}\hspace{-0.5\tmplenA}
	#1
	\hspace{-0.5\tmplenA}\raisebox{\tmplenB}{$\urcorner$}}
\newcommand{\god}[1]{\settowidth{\tmplenA}{$\ulcorner$}        
	\settoheight{\tmplenB}{$#1$}
	\setlength{\tmplenB}{1.1\tmplenB}
	\addtolength{\tmplenB}{-\tmplenA}
	\raisebox{\tmplenB}{$\ulcorner$}\hspace{-0.2\tmplenA}
	#1
	\hspace{-0.2\tmplenA}\raisebox{\tmplenB}{$\urcorner$}}
\newcommand{\is}{\mathrm{I}\vS}
\newcommand{\ind}[1]{\mathrm{I}{#1}}
\newcommand{\lind}[1]{(\mathrm{I}#1)^{-}}
\def\ti{\mathrm{TI}}
\newcommand{\rti}[1]{\mathrm{TI}_{#1}}
\def\rfn{\mathrm{RFN}}
\def\lrfn{\mathrm{Rfn}}
\newcommand{\rrfn}[1]{\mathrm{RFN}_{#1}}
\newcommand{\prog}{\mathrm{Prog}}
\newcommand{\wf}{\mathrm{WF}}

\def\epsnot{\varepsilon_0}

\newcommand{\con}{\mathrm{Con}}
\newcommand{\pr}{\mathrm{Pr}}
\newcommand{\tr}{\mathrm{Tr}}
\newcommand{\T}{\mathrm{T}}
\newcommand{\ax}{\mathrm{Ax}}
\newcommand{\ea}{\ensuremath{\mathsf{EA}}}         
 
\newcommand{\pa}{\ensuremath{\mathsf{PA}}}   
\newcommand{\rca}{\ensuremath{\mathsf{RCA}_0}} 
\newcommand{\wkl}{\ensuremath{\mathsf{WKL}_0}}
\newcommand{\aca}{\ensuremath{\mathsf{ACA}_0}}
\newcommand{\atr}{\ensuremath{\mathsf{ATR}_0}}
\newcommand{\ca}{\ensuremath{\mathsf{CA}}} 
\newcommand{\ac}{{\sf AC}}    
\newcommand{\imp}{\rightarrow}

\newcommand{\biimp}{\leftrightarrow}

\title[Uniform reflection in second order arithmetic]{A note on fragments of uniform reflection in second order arithmetic}
\thanks{The author's research was supported by the Alexander von Humboldt foundation.}

\author[Frittaion]{Emanuele Frittaion}
\address{Department of Mathematics, Technische Universit\"{a}t Darmstadt, Germany}

\begin{document}


\begin{abstract} 
We consider fragments of uniform reflection for formulas in the analytic hierarchy over theories of second order arithmetic. The main result is that  for any  second order arithmetic theory $T_0$ extending $\rca$ and  axiomatizable by a  $\Pi^1_{k+2}$ sentence,   and for any  $n\geq k+1$,
\[    T_0+ \rrfn{\vP^1_{n+2}}(T)  \ = \  T_0 + \rti{\vP^1_n}(\epsnot), \]
\[     T_0+ \rrfn{\vS^1_{n+1}}(T) \ = \      T_0+ \rti{\vP^1_n}(\epsnot)^{-},   \]
where $T$ is $T_0$ augmented with full induction, and $\rti{\vP^1_n}(\epsnot)^{-}$ denotes the schema of transfinite induction up to $\epsnot$ for $\vP^1_n$ formulas without set parameters. 
\end{abstract}

\maketitle

\section{Introduction}
The following benchmark results showcase the  relationship between uniform reflection and induction in the context of first order arithmetic:
\[   \tag{1} \ea+ \rfn(\ea)\ =\ \pa, \]
where $\ea$ is Kalm\'{a}r elementary arithmetic, and
\[   \tag{2} \pa+\rfn(\pa) \ =\ \pa+\ti(\epsnot). \]  
(1) is usually attributed to  \cite{KL68}. (2) is a special case of \cite[Theorem 12]{KL68}.

This note is concerned with  uniform reflection in the context of second order arithmetic.\footnote{ For an extensive study of iterated  uniform reflection in second order arithmetic, see  Pakhomov and Walsh \cite{PW21, PW22}.} It can be gleaned from 
Kreisel and L\'{e}vy \cite[Theorem 14]{KL68} that
\[    {\sf RCA} +\rfn({\sf RCA}) \ = \ {\sf RCA} +\ti(\epsnot), \]
where ${\sf RCA}$ is $\rca$ together with full induction.\footnote{Kreisel and L\'{e}vy prove the result for a slightly different system, denoted ${\sf Z}_1$, which appears in Howard and Kreisel \cite{HK66}. The theory ${\sf Z}_1$ consists of number and function variables, axioms for the constants zero, successor, pairing, projections,  axioms expressing the closure of all functions under composition and primitive recursion, and full induction (induction for all formulas in the language).} 

The same proof-theoretic methods of \cite{KL68}, namely, formalized cut elimination, show that equations (1) and (2) lift up  to any  theory $T$ of second order arithmetic of the form $T_0$ together with full induction, where $T_0$ is  any finitely axiomatizable  extension of $\rca$. We thus have (see Theorem \ref{thm uniform})
\[  \tag{3}  T_0 + \rfn(T_0) \ = \ T, \]
\[   \tag{4} T_0 + \rfn(T) \ = \ T_0 + \ti(\epsnot). \]
More in general (cf.\ \cite[Theorem 12]{KL68}), 
\[   T_0+ \rfn(T+\ti(\alpha)) \ =\ T_0 + \ti(\varepsilon(\alpha)), \]
where $\varepsilon(\alpha)$ is the least $\varepsilon$-number beyond $\alpha$.
The schema $\ti(\alpha)$ of transfinite induction along $\alpha$ consists of formulas
\[  \forall x\, (\forall y\prec x\, \vp(y)\imp \vp(x))\imp \forall x\, \vp(x), \]
where $\vp(x)$ ranges over all formulas and   $\prec $ is a primitive recursive ordinal notation system for  $\alpha$.  

Anyone familiar with (a bit of) ordinal analysis will readily realize that (3) and (4) hold  because of the unlimited  amount of induction, respectively transfinite induction up to $\epsnot$, available. The point we are trying to make is that the so-called $\vP^1_1$ proof-theoretic ordinals of $T_0$ and  $T$ may be way bigger than $\epsnot$.  Recall that the $\vP^1_1$ proof-theoretic ordinal $|S|$ of a second order arithmetic theory $S$ can be defined as the supremum of the provably well-founded  ordinals of $S$. These are  the ordinals $\alpha$ such that $S\vdash \wf(\alpha)$, where $\wf(\alpha)$ expresses the well-foundedness of $\alpha$ and corresponds to the following single instance of transfinite induction \[ \forall X\,  (\forall \beta\, (\forall \gamma<\beta\, (\gamma\in X) \imp \beta\in X)\imp \forall \beta<\alpha \, (\beta\in X)). \]
Actually, one always works with a given primitive recursive ordinal notation system, so that the above definition is relative to the choice of such  notation system (see \cite{R99} for an overview). 
Now, under certain conditions (cf.\ \cite[Proposition 2.13]{R99}), $|S|$ coincides with the least ordinal $\alpha$ such that \[ S+\wf(\alpha)\vdash \con(S).\] Again, the above definition refers to a given ordinal notation system.   In particular,  for theories of great proof-theoretic strength,  it   is certainly not the case that $T_0 + \wf(\epsnot) \vdash \con(T_0)$, let alone  $T_0 + \wf(\epsnot) \vdash \rfn(T)$. So much for that. 

A more subtle point is that the restriction to  finitely axiomatizable theories  is  a necessary one, as the result does not apply to recursively enumerable  theories $T_0$ of bounded complexity.
The assumption plays a role in the direction from (transfinite) induction to reflection.  The reason, roughly speaking, is that an infinite theory  may not prove that all its axioms are true. In fact, for sufficiently strong theories, such condition entails  finite 
axiomatizability. 
Note that main  subsystems of reverse mathematics, the so-called {\em Big Five},  are finitely axiomatizable \cite{S09}.  
We refer to Section \ref{discussion} for a discussion of this issue.

In this note we obtain the following refinement of  (3) and (4). We show in Theorem \ref{fragments} that if $T_0$  is any theory of second order arithmetic extending $\rca$ and axiomatized by a $\vP^1_2$  (in general, $\vP^1_{k+2}$) sentence, $T$ is $T_0$ plus full induction, and $n\geq 1$ (in general, $n\geq k+1$), then over $T_0$
	\[  \tag{5}  \rrfn{\vP^1_{n+2}}(T_0) \ = \ \ind{\vP^1_n}  \ \supseteq \  \lind{\vP^1_n} \ = \ \rrfn{\vS^1_{n+1}}(T_0),    \]
\[  \tag{6}  \rrfn{\vP^1_{n+2}}(T)  \ = \ \rti{\vP^1_n}(\epsnot)  \ \supseteq \ \rti{\vP^1_n}(\epsnot)^{-}  \ = \  \rrfn{\vS^1_{n+1}}(T).   \]

The superscript ${}^-$ refers to the restriction of the pertaining schema to formulas without set parameters.  
We  prove in Theorem \ref{strict inclusion} that, under certain conditions, the inclusions in (5) and (6) are strict. This is obtained by a standard application of G\"{o}del's second incompleteness theorem. 

In first order arithmetic we have a neat correspondence between uniform reflection and induction (see  Leivant \cite{L83} and Ono \cite{O87}):  for all $n\geq 1$,
\[   \ea +  \rrfn{\vP_{n+2}}(\ea)\ =\ \ea +\rrfn{\vS_{n+1}}(\ea)\ =\ \ea + \ind{\vP_n}.   \]
If anything, we can say that a less clear-cut picture emerges when we move to second order arithmetic.\footnote{It should be noted that parameter free versions of induction in first order arithmetic, considered, e.g., in Kaye, Paris, and  Dimitracopoulos \cite{KPD88} from a model-theoretic point of view,  can be characterized by using \emph{relativized} forms of local reflection (see Beklemishev \cite[Sect.\ 3]{B99}).} 
On the other hand (cf.\ \cite[Sect.\ 3.1]{PW22}), one could recover   the equivalence between  $\vP^1_{n+2}$  and $\vS^1_{n+1}$ uniform reflection  by considering a   relativized version of uniform reflection with set parameters.  More precisely, given a set $X$ of natural numbers and a second order arithmetic theory $T$, let $T(X)$  be the extension of $T$  with  {\em oracle axioms} $\bar n\in \bar X$ for $n\in X$ and $\bar n\notin \bar X$ for $n\notin X$, where $\bar X$ is a new set constant. Note that this construction makes sense in $\rca$ since the theory $T(X)$ is recursive in $T$ and $X$. One can then introduce the schema $\rfn^*(T)$ given by formulas 
\[  \forall x\, \forall X\, (\pr_{T(X)}(\godel{\vp(\dot x,\dot X)}) \imp \vp(x,X)).  \]
Now, the proofs of Theorem \ref{thm uniform} and Theorem \ref{fragments} go  through \emph{mutatis mutandis}. 
Note that
\[   \pr_T(\godel{\forall X\,\vp(X)}) \imp \forall X\,\pr_{T(X)}(\godel{\vp(\dot{X})}) \]
is provable in $\rca$. In particular, one obtains 
\begin{align*}
	T_0 + \rrfn{\vP^1_{n+2}}^*(T_0) &= \ T_0 + \ind{\vP^1_n}   \ = \  T_0 + \rrfn{\vS^1_{n+1}}^*(T_0),   \\
	 T_0 + \rrfn{\vP^1_{n+2}}^*(T)  & = \ T_0 + \rti{\vP^1_n}(\epsnot) \ =\ T_0 + \rrfn{\vS^1_{n+1}}^*(T), 
\end{align*}
where  $T_0$, $T$ and $n$ are as above.

\section{Basics}\label{basics}

\subsection{First order arithmetic}
For the purposes of this paper, let the language of first order arithmetic consist of finitely many symbols including $0,1,+,\times,=,<$. As usual, $\vD_0=\vP_0=\vS_0$ denotes the class of formulas built up from atomic formulas by means of boolean connectives and bounded quantifiers of the form $\exists x< t$ and $\forall x<t$, where $t$ is a term. The arithmetic hierarchy for $n>0$ is thus defined: $\vP_{n+1}=\{\forall x\, \vp(x)\mid \vp(x)\in \vS_n\}$  and dually $\vS_{n+1} =\{ \exists x\, \vp(x)\colon \vp(x)\in \vP_n\}$. 

\subsection{Reflection}
We assume a primitive recursive G\"{o}del numbering  $\god{\vp}$  of formulas. 
The uniform reflection principle $\rfn(T)$ for a theory $T$ is a schema consisting of formulas
\[      \forall x\,(\pr_T(\godel{\vp(\dot x)})\imp \vp(x)),  \]
where $\vp(x)$ is a formula with at most the displayed free variable and $\pr_T(y)$ is a canonical provability predicate for $T$ (cf.\ \cite{F60,F62}). The formula   $\pr_T(\godel{\vp(\dot x)})$ can be seen as  an abbreviation of \[\exists y\, (\pr_T(y)\land \sigma(\god{\vp},x,y)).\] Here,  $\sigma(u,x,y)$ is a $\vS_1$ formula defining  the primitive recursive operation  $(\god{\vp},n)\mapsto \godel{\vp(\bar n/v)}$, where $v$ is the only free variable of $\vp$ and $\vp(\bar n/v)$   denotes the substitution of the numeral $\bar n$ for  $v$ in $\vp$.  For ease of notation, we omit the overline notation for numerals when writing $\god{\vp}$ inside a formula.

\subsection{Second order arithmetic}
The language of second order arithmetic is  two-sorted, with
first order variables $x,y,z,\ldots$ intended to range over natural numbers, and
second order variables $X,Y,Z,\ldots$ intended to range over sets of natural numbers. It is obtained from the language of first order arithmetic by adding a membership relation symbol $\in$ connecting the two sorts. The equality relation symbol  $=$ only applies to first order terms.

The  classes of formulas $\vD^0_0$, $\vP^0_n$, $\vS^0_n$,  are defined just as in the first order case, but now set parameters are allowed. In this context, a formula is arithmetic if it does not contain set quantifiers, that is, quantifications of the form $\forall X$ and $\exists X$. Note that an arithmetic formula may contain free set variables. 
A formula  is  $\vP^1_n$  if it is of the form
$\forall X_1\, \exists X_2 \cdots Q X_n\,\vp$, 
where $\vp$ is arithmetic. The definition of $\vS^1_n$ formulas is dual.

The theory known as (full) second order arithmetic, often denoted ${\sf Z}_2$, is given by the usual first order axioms for zero, successor, addition, multiplication, and less than relation, together with  comprehension and induction schemata\smallskip
\[ \tag{Comprehension} \exists X\,\forall x\,(x\in X\leftrightarrow \vp(x)), \]
\[ \tag{Induction}  \vp(0)\land \forall x\,(\vp(x)\rightarrow \vp(x+1))\rightarrow\forall x\, \vp(x),  \]
where $\vp$ ranges over all formulas.
Note that every instance of  induction  can be obtained from comprehension and the induction axiom
\[    0\in X\land \forall x\,(x\in X\imp x+1\in X)\imp \forall x\,(x\in X). \]

The main subsystems of reverse mathematics \cite{S09}, also known as the Big Five, are obtained by prescribing how much induction one can use and what sets one can form
(in terms of comprehension or set existence axioms).   These are 
$\rca$ (Recursive Comprehension Axiom), $\wkl$ (Weak K\"{o}nig's Lemma), $\aca$ (Arithmetical Comprehension Axiom), $\atr$ (Arithmetical Transfinite Recursion) and $\vP^1_1$-$\ca_0$ (Comprehension for $\vP^1_1$ formulas).    $\rca$,  the standard base theory of reverse mathematics, consists of $\vD^0_1$ comprehension and $\vS^0_1$ induction. It is well known that $\rca$ is a conservative extension of $\is_1$, that is, Peano arithmetic $\pa$ with induction restricted to $\vS_1$ formulas. The theory $\aca$ consists of comprehension and induction for all arithmetic formulas.  

\subsection{Partial truth definitions}
In first order arithmetic one has access for any $n\geq 1$ to a  $\vS_n$ (resp.\ $\vP_n$)  partial truth definition for the class $\vS_n$ (resp.\ $\vP_n$) over $\is_1$. The construction involves a partial truth definition for $\vD_0$ formulas, of complexity $\vD_1$ over $\is_1$ (cf.\ \cite[Ch.\ 1, Sect.\ 1(d)]{HP98}). By the same token, one can construct partial truth definitions for formulas of bounded complexity, say of bounded rank, over $\is_1$ (cf.\ \cite[1.5.4]{T73}). 

In the context of second order arithmetic, one can define truth predicates for arithmetic formulas and $\vP^1_n$ formulas but they fail to be provably so over weak fragments such as $\rca$. However, for the purposes of this paper, it will be sufficient to work with partial truth definitions for  formulas of bounded complexity. Let $R_m$ denote  the class of formulas of rank at most $m$ and let $P_{n,m}$ denote the (closure under subformulas of the) class of $\vP^1_n$ formulas with arithmetic matrix of rank at most $m$. We define the rank as usual: the rank is a natural number and the rank of a compound formula is bigger than the rank of its components. Notice that we make no distinction between number and set quantifiers. The construction of partial truth definitions in first order arithmetic easily extends to arithmetic formulas with set parameters.  
In particular,  for any choice of $n$ and $m$, one obtains a partial truth predicate for $R_m$ and a $\vP^1_n$ truth predicate for $P_{n,m}$  formulas over $\rca$. 

We will use partial truth definitions of the first kind (for $R_m$)  in the proof of Theorem \ref{thm uniform}. In Theorem \ref{fragments}  we will choose a $\vP^1_n$ partial truth definition for  sufficiently many $\vP^1_n$ formulas, namely, for  a sufficiently large  class of the form  $P_{n,m}$.

At one point (see Lemma \ref{strict}) we will use the fact that  $\vP^1_n$ (resp.\ $\vS^1_n$) uniform reflection can be finitely axiomatized over $\aca$ for each $n\geq 1$. This can be attained by means of universal formulas.   We say that a $\vP^1_n$  formula $\vp_U(e,x)$ with displayed free variables is {universal} $\vP^1_n$  in a theory $T$ if for all $\vP^1_n$  formulas $\vp(x)$, $T$ proves $\forall x\, (\vp(x)\biimp \vp_U(\god{\vp}, x))$.  
Universal $\vS^1_n$ formulas are defined in the exact same way. The existence in $\aca$ of universal $\vP^1_n$ and $\vS^1_n$ formulas for every $n\geq 1$ is basically folklore. There is some leeway in designing universal formulas. For example, one possible route involves the following two steps in the construction of a universal $\vP^1_1$ formula with set parameters; it is then routine to  build up universal $\vP^1_n$  and $\vS^1_n$ formulas for all $n\geq 1$.  First, given a $\vS^0_1$ formula $\vp(x,f,X_1,\ldots X_k)$, where $f$ denotes a function from $\omega$ to $\omega$,  one can  primitive recursively find an index  $e$ such that
\begin{align*} \vp(x,f,X_1,\ldots,X_k) & \biimp \{e\}^{f\oplus X_1\oplus\cdots \oplus X_k}(x)\downarrow\, \\  & \biimp \exists s\, \vartheta(e,x,s,\bar f(s), \bar X_1(s),\ldots,\bar X_k(s)), \end{align*} 
\sloppy  where $\vartheta(e,x,s,z)$ is a fixed $\vD^0_0$ formula, $\bar f(s)$ encodes the sequence ${\langle f(0),\ldots,f(s-1)\rangle}$ and $\bar X_i(s)$ encodes the first $s$ bits of the characteristic function of $X_i$. This version of the enumeration theorem can be formalized in $\rca$ (cf.\ \cite[Theorem\ II.2.7, p.\ 68]{S09}). Second, given a $\vP^1_1$ formula $\vp(x,X_1,\ldots,X_k)$, one can primitive recursively construct a $\vS^0_1$ formula $\vp_0(x,f,X_1,\ldots,X_k)$ such that $\vp(x,X_1,\ldots,X_k)\biimp \forall f\, \vp_0(x,f,X_1,\ldots,X_k)$. This is  Kleene normal form theorem for $\vP^1_1$ formulas and can be formalized in $\aca$ (cf.\  \cite[Lemma V.1.4, p.\ 169]{S09}). Another  approach (cf.\ \cite[Sect.\ 2.1]{PW21}) is to further formalize the second step with the help of a partial truth definition for $\vS^0_1$ formulas. 

We refer the reader to \cite[Sect.\ 2.1]{PW21} for more background on partial truth predicates in second order arithmetic.


\section{The uniform reflection principle in analysis}\label{sec uniform}
In this section we outline a proof of the following theorem. 

\begin{theorem}[Essentially, Kreisel and L\'{e}vy \cite{KL68}] \label{thm uniform}
Let $T_0\supseteq\rca$ be a finitely axiomatizable second order arithmetic theory.
 Let $T$ be $T_0$ plus the schema of full induction. Then 
\begin{equation*}
 \tag{3} T_0 +\rfn(T_0)= T.
\end{equation*}
\begin{equation*}
 \tag{4} T_0 + \rfn(T)= T_0 + \ti(\epsnot). 
\end{equation*}
\end{theorem}
\begin{proof}
(3) For the forward direction note that for every standard $n$ the formula
\[        \vp(0)\land \forall x\,(\vp(x)\imp \vp(x+1)) \imp  \vp(\bar n) \]
is provable in classical logic, and hence in any theory whatsoever. The construction of such a proof is primitive recursive (indeed elementary recursive) in $n$.  This can be formalized in $\rca$. Therefore
\[   T_0\vdash \forall x\, \pr_{T_0}(\godel{\vp(0)\land \forall x\,(\vp(x)\imp \vp(x+1)) \imp  \vp(\dot x)}). \]
It follows that 
\[   T_0 +\rfn(T_0)\vdash \forall x\,(\vp(0)\land \forall x\,(\vp(x)\imp \vp(x+1)) \imp  \vp(x)), \]
which is classically equivalent to induction for $\vp$. 

For the other direction, we use cut elimination for classical logic, which is available in $\rca$. More in detail, fix a sentence $\psi$  axiomatizing  $T_0$ and let $\vp(x)$ be  given . We aim to show 
\[   T\vdash \pr_{T_0}(\godel{\vp(\dot x)})\imp \vp(x). \]
For convenience, let us consider a Tait style sequent calculus for classical logic.\footnote{Tait one-sided calculi for first order classical logic (with equality) and cut elimination thereof are presented  in e.g.\ \cite[Sect.\ 3.6]{TS20} and \cite[Sects.\ 1.2, 2.1.2 and 3.1.1]{A20}. Here, of course, we are dealing with a two-sorted variant of any such  calculus.}  
We reason informally in $T$. If $\vp(\bar n)$ is provable in $T_0$, then  there is a finite cut-free proof of the sequent $\neg\psi, \vp(\bar n)$.   By induction on the height of the proof, one shows that every sequent in the proof is true.  Now, $\psi$ is true, and so $\vp(\bar n)$ must be true.   The induction argument can be formalized in $\sf RCA$ with the aid of a partial truth predicate for $R_m$ such that both $\neg \psi$ and $\vp(x)$ are in $R_m$. Note in fact that by the subformula property a cut-free proof  consists only of subformulas of the end sequent.

(4) The proof that $T_0 + \rfn(T)$ entails $\ti(\epsnot)$ is based on a straightforward generalization of Gentzen's lifting result \cite[p.\ 293]{G69} (i.e., the closure of provable transfinite induction under $\alpha\mapsto\omega^\alpha$ over $\pa$). Let $\prog(\vp)$ be a shorthand for $\forall x\,(\forall y\prec x\,\vp(y)\imp \vp(x))$.  Then for  every formula $\vp(x)$ and for every standard  number $k$,  $T$ proves 
\[    \prog(\vp)\imp \forall x\prec \omega_k\, \vp(x), \]
where $\omega_0=1$ and $\omega_{k+1}=\omega^{\omega_k}$. The construction of such a proof is uniform and gives a primitive recursive function that accepts a  $\vP^1_n$ formula $\vp(x)$ and a standard number $k>0$ as inputs and outputs a proof in $\rca$ plus induction for $\vP^1_{n+k}$ formulas of $\prog(\vp)\imp \forall x\prec\omega_{k+1}\,\vp(x)$.   This formalizes in $\rca$. The upshot is that
\[   T_0\vdash \forall z\, \pr_T(\godel{\prog(\vp)\imp \forall x\prec \omega_{\dot z}\, \vp(x)}). \]
Therefore 
\[  T_0 +\rfn(T)\vdash \forall z\,(\prog(\vp)\imp \forall x\prec \omega_z\, \vp(x)). \]
The conclusion follows.
 
The converse direction (cf.\  the proof of \cite[Theorem 14]{KL68}) is based on the fact that in  a suitable infinitary system for $\omega$-arithmetic one can eliminate cuts and then prove by transfinite induction on
$\epsnot$ that every sequent in a cut-free $\omega$-proof is true. 

More in detail. Let $\psi$ be a sentence axiomatizing $T_0$. Given a formula $\vp$, we aim to show  	
\[  T_0 + \ti(\epsnot)\vdash    \pr_T(\godel{\vp(\dot x)})\imp \vp(x). \]   
Again, let us consider a two-sorted extension of a finitary Tait sequent calculus for $\pa$, call it $\pa^2$, and its infinitary counterpart with the $\omega$-rule, call it $\pa^2_\infty$.\footnote{Tait one-sided calculi for $\pa$ and its infinitary counterpart, call it $\pa_\infty$, are described in \cite{B97, B91}.  One obtains a two-sorted extension by simply adding the appropriate axioms  and set quantifier rules to such calculi (cf.\ \cite[Sect.\ 3.1]{PW22}).}
We reason informally in $T_0 + \ti(\epsnot)$. 
If $T$ proves $\vp(\bar n)$, then there is a proof of  the sequent  $\neg \psi,\vp(\bar n)$ in $\pa^2$. 
We first convert such finite proof  into an $\omega$-proof  of $\neg \psi,\vp(\bar n)$ in $\pa^2_\infty$ of height $<\omega\cdot 2$ and cut formulas of rank  $<r$, for some natural number $r$.\footnote{The bound $\omega\cdot 2$ arises from a sequent style formalization of Peano arithmetic with induction axioms instead of rules for induction. Induction rules would yield the bound $\omega^2$. The rank is defined as usual. Literals (atomic and negated atomic formulas) receive rank $0$, $\rank(\vp\circ\psi)=\max(\rank(\vp),\rank(\psi))+1$ for $\circ\in\{\lor,\land\})$, $\rank(\circ\, x\, \vp(x))=\rank(\vp(x))+1$ and $\rank(\circ\,  X\, \vp(X))=\rank(\vp(X))+1$ for $\circ \in\{\exists,\forall\}$.}  By applying full cut elimination, we then obtain a cut-free $\omega$-proof of height $<\epsnot$.  We can now show that every sequent in the proof is true by transfinite induction on $\epsnot$. As before, since $\psi$ is true,  $\vp(\bar n)$ must be true. 
This concludes the proof. 

Let us notice that the   embedding  of $\pa$ into an infinitary system with the $\omega$-rule, call it $\pa_\infty$,  and cut elimination thereof (see, e.g., \cite[Theorems 28.5 and 22.8]{Sch77}) extends to   $\pa^2$ and $\pa^2_\infty$ in a  straightforward manner. On the other hand, it is by now clear how to formalize all this   in e.g.\ $\rca$ by considering $\omega$-proofs with ordinal tags below $\varepsilon_0$.\footnote{The reader should inspect  Kreisel and L\'{e}vy's {\em hints and tips} explanation on how to deal with  infinite proof trees and  cut elimination (see \cite[pp.\ 127, 128]{KL68}).} 
In particular, $\omega$-proofs can be directly represented as infinite trees (see Mints \cite{M78}, Friedman and Sheard \cite{FS95}) or  suitably coded  by numbers (see Schwichtenberg \cite{S77} and  Buchholz \cite{B91}). What matters is that being  (a code of) an $\omega$-proof of height less than $\epsnot$ can be expressed in an arithmetic way, whether this is a property about sets or about numbers (it can be $\vP^0_1$ \cite{S77} or even primitive recursive  \cite{B91}). We point out that for the purposes of cut elimination a key element of any such coding appears to be  Mint's repetition rule \cite{M78}.
Finally, the argument by  transfinite induction can be formalized by using a partial truth predicate as before, thanks to the fact that cut-free proofs (provably) enjoy the subformula property.
\end{proof}

\section{Some remarks}\label{discussion}
The proof of uniform reflection from (transfinite) induction relies on  partial truth definitions and the assumption that $T_0$ is finitely axiomatizable. 
We have already addressed in Section \ref{basics} how to deal with partial truth definitions over $\rca$. 
With regard to finite axiomatizability, it is clear that any sufficiently strong finite theory can prove its own truth. In fact,  if $T_0\supseteq \rca$ is axiomatized by a sentence $\psi$, then 
\[ T_0\vdash \ax_{T_0}(x)\imp \tr(x), \]
by letting $\ax_{T_0}(x)$ be $x=\god{\psi}$, where $\tr(x)$ is a partial truth definition for a sufficiently large class of sentences including $\psi$.   Conversely, under mild conditions, every theory capable of proving its own truth must be finitely axiomatizable.
\begin{proposition}\label{truth}
Let $T\supseteq S$ be a recursively enumerable theory and $S\supseteq\rca$ be finitely axiomatizable. Suppose there is predicate $\T(x)$ such that 
\[ T\vdash \ax_T(x)\imp \T(x), \]
where $\ax_T$ is a $\vS_1$ definition of $T$, and suppose that 
\[ S\vdash \T(\god{\vp})\imp \vp, \]
for every axiom $\vp\in T$. Then $T$ is finitely axiomatizable.
\end{proposition}
\begin{proof}
The theory  $S +\{\forall x\,(\ax_T(x)\imp \T(x))\}$ is finitely axiomatizable and equivalent to $T$.  Note that, by $\vS_1$ completeness,  $\rca\vdash \ax_T(\god{\vp})$ for every $\vp\in T$. 
\end{proof}
In general, Theorem \ref{thm uniform} fails for infinite theories of bounded complexity (e.g., axioms with a bound on the number of set quantifiers or even axioms with bounded rank).  
The following simple example is a case in point. 

\begin{example}
Let $T_0$  be $\rca  +\{ \emptyset^{(\bar n)} \text{ exists }\colon n\in\omega\}$. Note that $\rca$ is finitely axiomatizable and hence $T_0$ is equivalent to a theory of bounded rank.\footnote{One can express the existence of the $n$-th jump by a single formula $\exists Z\, \exists X\, ((X)_0=\emptyset\land (\forall m<n)\, (X)_{m+1}=(X)'_m\land Z=(X)_n)$, where $(X)_m=\{n\mid (m,n)\in X\}$ is the $m$-th column of $X$.} We claim that $T_0 +\ti(\epsnot)$ does not prove uniform reflection over $T_0$. In fact,  $T_0\vdash \forall x\,\pr_{T_0}(\godel{\emptyset^{(\dot x)} \text{ exists}})$.
By reflection one would obtain 
\[  T_0 +\ti(\epsnot)\vdash \forall x\,(\emptyset^{(x)} \text{ exists}). \]
A standard compactness argument shows that there is a model of $T_0 +\ti(\epsnot)$ where $\forall x\,(\emptyset^{(x)} \text{ exists})$ fails. In the compactness argument use the fact that for every $n$ the $\omega$-model $\{X\subseteq\omega\colon X\leq_T \emptyset^{(n)}\}$ is a model of $\rca + \ti(\epsnot)$ that satisfies
\[   \emptyset^{(\bar n)} \text{ exists and }  \exists x\,(\emptyset^{(x)} \text{ does not exist}). \]
\end{example} 


\section{Fragments}\label{sec fragments}
Let  $\rrfn{\vP^1_n}(T)$ be the restriction of $\rfn(T)$ to $\vP^1_n$ formulas. The schema $\rrfn{\vS^1_n}(T)$  is defined similarly.  
Observe that the uniform reflection schema applies to formulas $\vp(x)$ with no free variables other than $x$. 
Therefore, the schemata $\rrfn{\vP^1_n}(T)$  and $\rrfn{\vS^1_n}(T)$ refer to  $\vP^1_n$ and $\vS^1_n$ formulas $\vp(x)$ with no free variables, in particular no free set variables, other than $x$.  However, we can apply uniform reflection to formulas with finitely many number variables. 
Note that in first order arithmetic 
the schema
\[   \pr_T(\godel{\vp(\dot x_1,\ldots,\dot x_k)})\imp \vp(x_1,\ldots, x_k) \]
is equivalent (over $\ea$) to the schema with only one variable (see, e.g., Feferman \cite{F62}).  The same applies to restrictions of uniform reflection in the $\vP_n$ and $\vS_n$ hierarchy. 
Similarly,   $\vP^1_n$ uniform reflection is equivalent (over $\rca$) to its multivariate version.  The same holds for $\vS^1_n$ uniform reflection. For example, given a $\vP^1_n$ formula $\vp(x_1,\ldots,x_k)$ with the free variables shown, consider the formula 
\[   \psi(x)\ =_{\mathrm{def}}\   \forall x_1\cdots\forall x_k\,((x)_1=x_1\land\cdots\land (x)_k=x_k\imp \vp(x_1,\ldots,x_k)). \]
Then  $\psi(x)$ is $\vP^1_n$ in $\rca$, 
\[ \rca\vdash \forall x_1\cdots\forall x_k\,\pr_T(\godel{\vp(\dot x_1,\ldots,\dot x_k)})\biimp \forall x\,\pr_T(\godel{\psi(\dot x)}), \]
and
\begin{align*}
 \rca & \vdash   \forall x_1\cdots\forall x_k\,(\pr_T(\godel{\vp(\dot x_1,\ldots,\dot x_k)})\imp \vp(x_1,\ldots,x_k))\biimp \\  & \biimp \forall x\,(\pr_T(\godel{\psi(\dot x)})\imp \psi(x)).
\end{align*}


For a fine characterization of uniform reflection, we need to consider lightface versions of induction and transfinite induction up to $\epsnot$. Let $\lind{\vP^1_n}$ be  the restriction of induction to $\vP^1_n$ formulas with no set parameters. Same definition applies to  $\rti{\vP^1_n}(\epsnot)$. 

\begin{theorem}\label{fragments}
Let  $T_0\supseteq\rca$ be a second order arithmetic theory  axiomatized by a $\vP^1_2$ sentence, and let $n\geq 1$. Let $T$ denote $T_0$ plus the schema of full induction. Over $T_0$, 
	\[  \tag{5}  \rrfn{\vP^1_{n+2}}(T_0) \ = \ \ind{\vP^1_n}  \ \supseteq \  \lind{\vP^1_n} \ = \ \rrfn{\vS^1_{n+1}}(T_0),    \]
	\[  \tag{6}  \rrfn{\vP^1_{n+2}}(T)  \ = \ \rti{\vP^1_n}(\epsnot)  \ \supseteq \ \rti{\vP^1_n}(\epsnot)^{-}  \ = \  \rrfn{\vS^1_{n+1}}(T).   \]
More in general, if $T_0$ is axiomatized by a $\vP^1_{k+2}$ sentence, then the above relations hold for all $n\geq k+1$.
\end{theorem}

\begin{proof}
We sketch a proof of (6).  \vspace{1mm}

1. From ($\vS^1_{n+1}$) $\vP^1_{n+2}$ uniform reflection to (lightface) $\vP^1_n$ transfinite induction. Note that for a $\vP^1_n$ formula $\vp(x)$, the formula 
\[    \psi(z)\ =_{\mathrm{def}}\  \forall x\,(\forall y\prec x\, \vp(y)\imp \vp(x))\imp \forall x\prec \omega_z\, \vp(x) \]
is $\vS^1_{n+1}$ within $\rca$ (by simple quantifier manipulations). \vspace{1mm}

1.1. Let us consider the lightface case. Suppose that $\vp(x)$  has no set parameters. Note that $\vp(x)$ may contain free number variables other than $x$, say $x_1,\ldots,x_i$. We have 
\[   T_0\vdash \forall z\,\forall x_1\cdots\forall x_i\, \pr_T(\godel{\psi(\dot z,\dot x_1,\ldots,\dot x_i)}). \] 
Then one can apply (the multivariate version of) $\vS^1_{n+1}$ uniform reflection, and obtain
\[  T_0 + \rrfn{\vS^1_{n+1}}(T) \vdash \forall z\,\forall x_1\cdots \forall x_i\, \psi(z,x_1,\ldots,x_i). \]

1.2. For the boldface version, suppose that $\vp(x)$  has parameters $x_1,\ldots,x_i$ and $X_1,\ldots,X_j$. Then   one must apply (the multivariate version of) uniform reflection to 
\[ \forall X_1\cdots\forall X_j\,\psi(z,x_1,\ldots,x_i,X_1,\ldots,X_j), \]
which is $\vP^1_{n+2}$. Therefore,
\[  T_0\vdash  \forall z\,\forall x_1\cdots \forall x_i\, \pr_T(\godel{\forall X_1\cdots \forall X_j\,\psi(\dot z,\dot x_1,\ldots,\dot x_i,X_1,\ldots,X_j)}), \]
and hence
\[  T_0 + \rrfn{\vP^1_{n+2}}(T)\vdash \forall z\,\forall x_1\cdots \forall x_i\,\forall X_1\cdots\forall X_j\,\psi(z,x_1,\ldots,x_i,X_1,\ldots,X_j). \]

\vspace{1mm}

2. From $\vP^1_n$ transfinite induction to $\vP^1_{n+2}$ uniform reflection. 

Fix a $\vP^1_2$ axiomatization $\forall X\,\psi(X)$ of $T_0$. Let $\forall X\,\exists Y\,\vp(x,X,Y)$ be a $\vP^1_{n+2}$ formula with no free variables other than $x$. By definition, $\vp(x,X,Y)$ is $\vP^1_n$. We aim to show
\[  T_0 + \rti{\vP^1_n}(\epsnot)\vdash \pr_T(\godel{\forall X\,\exists Y\,\vp(\dot x,X,Y)})\imp \forall X\,\exists Y\,\vp(x,X,Y). \]
Work in $T_0 +\rti{\vP^1_n}(\epsnot)$. Suppose that $\forall X\,\exists Y\,\vp(\bar x,X,Y)$ is provable in $T$.  By embedding and cut elimination, we obtain a cut-free $\omega$-proof of height $<\epsnot$ of the sequent
\[    \Xi\ =\  \exists X\,\neg \psi(X),  \forall X\,\exists Y\, \vp(\bar x,X,Y). \]

We aim to prove  that $\forall X\,\exists Y\,\vp(\bar x,X,Y)$ is true. Suppose, towards a contradiction, that there is a set  $X_0$  such that $\neg \exists Y\, \vp(x,X_0,Y)$ is true. We are going to argue as in Leivant \cite{L83} (cf.\ also \cite[Lemma 4.3]{O87}) and note that in order to analyze a cut-free proof of a false  $\vP^1_{n+2}$ formula we only need  a partial truth predicate for (enough) $\vP^1_n$ formulas. 
 
By the subformula property, every sequent $\Gamma$ appearing in a cut-free $\omega$-proof of $\Xi$ is of the form $\Pi,\Lambda$, where $\Pi$ consists of $\vP^1_n$ formulas and  $\Lambda\subseteq \Xi \cup\{\exists Y\, \vp(\bar x,U,Y)\mid U \text{ free variable}\}$.  Call a sequent of this form a $\Xi$-sequent. Now, given a cut-free $\omega$-proof of height $<\epsnot$ of a $\Xi$-sequent, we prove by $\vP^1_n$ transfinite induction on $\epsnot$ that:

\lq\lq For every
$\Xi$-sequent $\Pi,\Lambda$ appearing in such proof, the formula $\bigvee \Pi$ is true for any evaluation of the free variables, with the proviso that  the eigenvariables of a $\forall$-introduction of $\forall X\,\exists Y\,\vp(\bar x,X,Y)$  are evaluated by the set $X_0$."
		 
The property in quotation marks can be expressed by a $\vP^1_n$ formula.  For this we use a $\vP^1_n$ truth predicate for a sufficiently large class of $\vP^1_n$ formulas, one of the form $P_{n,m}$ so that both $\neg\psi(X)$ and $\vp(x,X,Y)$ are in $P_{n,m}$. 	 Note also that  $\vP^1_n$ induction suffices to prove that $\vP^1_n$ formulas  are closed under bounded quantifiers, so that the segment \lq\lq \ldots the formula $\bigvee \Pi$ is true\ldots \rq\rq\ is indeed equivalent to a $\vP^1_n$ formula. 

We consider two key cases. Suppose  we have an inference of the form
\[  \rul{\Pi,\neg \psi(W),\Lambda}{\Pi,\Lambda,\exists X\,\neg \psi(X)}{} \]
Here, $\neg \psi(W)$ is $\vP^1_1$ and hence $\vP^1_n$. Given an evaluation of all free variables as above, we want to show that $\bigvee \Pi$ is true. By possibly extending the evaluation to $W$, say by the set $X$, we can assume by  the induction hypothesis that $\bigvee \Pi\lor \neg\psi(X)$ is true.  By the assumption $\forall X\,\psi(X)$ we can clearly rule out $\neg \psi(X)$ from being true. Hence we are done.

The other interesting case is when we have an inference of the form 
\[ \rul{\Pi,\vp(\bar x,U,V),\Lambda}{\Pi,\Lambda,\exists Y\,\vp(\bar x,U,Y)}{} \]
Fix an evaluation of all variables such that the set variable $U$ is evaluated by the set $X_0$. If necessary, we can extend the evaluation to the set variable $V$, say by the set $Y$. By the induction hypothesis, $\bigvee \Pi\lor \vp(\bar x,X_0,Y)$ is true. On the other hand, from the assumption, $\vp(\bar x,X_0,Y)$ is not true. Thereby, $\bigvee \Pi$ is true.

We now have a contradiction since the $\Pi$ part of the sequent $\Xi$ is empty or possibly consisting of $\exists X\,\neg \psi(X)$. In each case, the sequent fails to satisfy the required property. \vspace{1mm}

3. From lightface $\vP^1_{n}$ transfinite induction to $\vS^1_{n+1}$ uniform reflection.

We  work  in $T_0 + \rti{\vP^1_n}(\epsnot)^{-}$.  Suppose that $\exists X\, \forall Y\, \vp(\bar x,X,Y)$, where $\vp(x,X,Y)$ is $\vS^1_{n-1}$, is provable in $T$. Then we have a cut-free $\omega$-proof of height $<\epsnot$ of the sequent
\[    \Xi\ =\  \exists X\,\neg \psi(X),  \exists X\, \forall Y\, \vp(\bar x,X,Y). \]
We now say that a $\Xi$-sequent is  one of the form $\Sigma,\Lambda$, where $\Sigma$ consists of $\vS^1_{n-1}$ formulas, and $\Lambda\subseteq \Xi\cup\{\forall Y\, \vp(\bar x, U, Y)\mid U \text{ free variable}\}$. Suppose, towards a contradiction, that $\neg \exists X\, \forall Y\, \vp(x,X,Y)$. Again, we follow Leivant's \emph{strategy} by considering only  true $\vS^1_{n-1}$ formulas in analyzing a cut-free proof of a false $\vS^1_{n+1}$ formula.  As before, every sequent in a cut-free $\omega$-proof of $\Xi$ is a $\Xi$-sequent. We want to show that:

\lq\lq For every $\Xi$-sequent $\Sigma,\Lambda$ of a cut-free  $\omega$-proof, the formula $\bigvee \Sigma$ is true 
for every evaluation of the free  variables."

A moment's reflection shows that this can be formalized by a $\vP^1_n$ formula with no set parameters. Note that if we use trees (second order objects) to represent $\omega$-proofs, we must quantify over all possible cut-free $\omega$-proofs in order to apply lightface $\vP^1_n$ transfinite induction. Let us consider the only key case. Suppose we have an inference of the form 
\[  \rul{\Sigma,\vp(\bar x, U,V),\Lambda }{\Sigma,\Lambda,\forall Y\, \vp(\bar x,U,Y)}{} \]
where $V$ is the eigenvariable, and so does not appear in the lower sequent. Given an evaluation of the free variables by sets $X,\ldots$, where $X$ is the evaluation of $U$, we want to show that $\bigvee \Sigma(X,\ldots)$ is true. By the induction hypothesis, for any  evaluation of the extra variable $V$, say by the set  $Y$, we have that $\bigvee\Sigma(X,\ldots)\lor \vp(\bar x, X,Y)$ is true. By the eigenvariable condition, this implies that $\bigvee\Sigma(X,\ldots) \lor \forall Y\, \vp(\bar x, X,Y)$ is true. By the assumption, we can rule out the case where $\forall Y\, \vp(\bar x, X,Y)$ is true, as desired. As before, one obtains a contradiction since the end sequent fails to satisfy the required property.
\end{proof}

\subsection{Separation results}
The local reflection principle $\lrfn(T)$ for a theory $T$ is the schema consisting of sentences
\[      \pr_T(\god{\vp})\imp \vp.  \]
We now  show, under mild assumptions, that the inclusions in Theorem \ref{fragments} are strict.  This will be an immediate consequence of the following.
\begin{lemma}\label{strict}
Let $T_0$ be a   $\vP^1_2$ finitely axiomatizable theory  extending $\aca$ and  $n\geq 1$.   Then 
\[ \tag{7} T_0  +  \rrfn{\vP^1_{n+2}}(T_0)\not\vdash \lrfn_{\vS^1_{n+2}}(T_0). \]
If, moreover, $T_0 + \rrfn{\vS^1_{n+1}}(T_0) + \vS^1_{n+1}$-$\ac$ is consistent, then 
\[ \tag{8} T_0 + \rrfn{\vS^1_{n+1}}(T_0) \not\vdash \lrfn_{\vP^1_{n+1}}(T_0). \]
The same holds with respect to reflection over $T$ instead of $T_0$, where $T$ is $T_0$ plus the schema of full induction.
\end{lemma}
\begin{proof}
By using a universal $\vP^1_{n+2}$ formula one can see that, over $\aca$, the schema $\rrfn{\vP^1_{n+2}}(T_0)$ can be axiomatized by a $\vP^1_{n+2}$ sentence, call it $\psi$. 
Now, \[T_0 + \lrfn_{\vS^1_{n+2}}(T_0)\vdash \psi\imp \neg\pr_{T_0}(\godel{\neg\psi}).\] 
By the second incompleteness theorem, $T_0 +\rrfn{\vP^1_{n+2}}(T_0)\not\vdash\lrfn_{\vS^1_{n+2}}(T_0)$. 
This proves (7). 

Let us show (8). By using a universal $\vS^1_{n+1}$ formula,  the schema $\rrfn{\vS^1_{n+1}}(T_0)$ can be axiomatized in $\aca$ by a sentence of the form $\forall x\,\vartheta(x)$, where $\vartheta$ is $\vS^1_{n+1}$. This formula is not $\vS^1_{n+1}$ yet. Now, let $\vp$ be the $\vS^1_{n+1}$ sentence obtained by applying choice to $\forall x\,\vartheta(x)$. By the consistency assumption, $T_0 +\{\vp\}$ is consistent. As before, 
\[ T_0 +\lrfn_{\vP^1_{n+1}}(T_0)\vdash  \vp \imp  \neg\pr_{T_0}(\god{\neg\vp}).\] 
By the second incompleteness theorem, $T_0 +\{\vp\}\not\vdash\lrfn_{\vP^1_{n+1}}(T_0)$. Note that $\vp$ implies $\forall x\,\vartheta(x)$, over, say, $\rca$. It follows that $T_0 +\rrfn{\vS^1_{n+1}}(T_0)\not\vdash \lrfn_{\vP^1_{n+1}}(T_0)$.
\end{proof}

By combining Theorem \ref{fragments} with Lemma \ref{strict} we then obtain the following. 
\begin{theorem}\label{strict inclusion}
Under the hypotheses of Lemma \ref{strict},
\[  \tag{5$^\neg$} T_0 + \ind{\vP^1_n} \not\vdash \lind{\vP^1_{n+1}}, \]
\[  \tag{6$^\neg$} T_0 + \lind{\vP^1_{n+1}} \not\vdash \ind{\vP^1_n}. \]
The same holds with respect to $\rti{\vP^1_n}(\epsnot)$ and $\rti{\vP^1_{n+1}}(\epsnot)^{-}$. 
\end{theorem}

\begin{remark}
The proof of Theorem \ref{fragments} shows that \[T_0 +(\ind{\vP^1_n})^{--}\vdash \lrfn_{\vS^1_{n+1}}(T_0), \] 
where ${}^{--}$ denotes the restriction of induction to $\vP^1_n$ formulas with no parameters at all. Hence, we can  strengthen (5$^\neg$) by replacing $\lind{\vP^1_{n+1}}$ with its ${}^{--}$ sibling. Similarly,
\[  T_0 + \rti{\vP^1_n}(\epsnot)^{--}\vdash \lrfn_{\vS^1_{n+1}}(T), \]  
and hence $T_0 +\rti{\vP^1_n}(\epsnot) \not\vdash \rti{\vP^1_{n+1}}(\epsnot)^{--}$.
\end{remark}

\begin{question}
	Can we drop the assumptions of Lemma \ref{strict}? What is the relation between local reflection, induction, transfinite induction up to $\epsnot$, and corresponding parameter free variants ${}^-$ and ${}^{--}$?
\end{question}

\section*{Acknowledgments}
We would like to thank the referees for their thorough and detailed reports.


\bibliographystyle{plain}

\end{document}